\documentclass[11pt]{article}

\usepackage{cite}
\usepackage{subfigure}
\usepackage{graphicx, color, graphpap}
\usepackage[normalem]{ulem}
\usepackage{amsmath,amssymb,amsthm}
\usepackage{ifthen}
\usepackage{mathrsfs}
\usepackage{mathtools}
\usepackage{enumitem}
\usepackage{lineno}
\usepackage{float}
\usepackage{fancyhdr}
\usepackage[us,12hr]{datetime} 
\usepackage{exscale}
\usepackage{tabularx}
\usepackage{latexsym}
\usepackage{fullpage}       
\usepackage{verbatim}
\usepackage{multirow}
\usepackage{overpic}
\usepackage{comment} 
\usepackage{hyperref}
\usepackage{blkarray}
\usepackage{adjustbox}
\usepackage{tcolorbox}

\usepackage{pgfplots}
\usepackage{marvosym}
\usepackage[noabbrev]{cleveref}
\usepackage{tikz}

\usepackage{algorithm}
\usepackage{algpseudocode}

\pgfplotsset{compat=1.11}
\usetikzlibrary{arrows, arrows.meta}
\usetikzlibrary{shapes.geometric, arrows.meta, positioning}

\tikzstyle{startstop} = [rectangle, rounded corners, minimum width=3.5cm, minimum height=1cm,text centered, draw=black, fill=blue!10]
\tikzstyle{process} = [rectangle, minimum width=4cm, minimum height=1cm, text centered, draw=black, fill=green!10]
\tikzstyle{io} = [trapezium, trapezium left angle=70, trapezium right angle=110, minimum width=3.5cm, minimum height=1cm, text centered, draw=black, fill=orange!15]
\tikzstyle{arrow} = [thick,->,>=stealth]

\definecolor{methD}{RGB}{230,120,50}
\definecolor{methDD}{RGB}{50,140,90}
\definecolor{methSDD}{RGB}{50,90,170}
\definecolor{methPSD}{RGB}{180,40,40}
\tikzset{
	visbox/.style={rounded corners=8pt, line width=1.2pt, align=center, inner sep=8pt},
	visnestouter/.style={visbox, minimum width=9.6cm, minimum height=4.4cm},
	visnest/.style={visbox, minimum width=7.6cm, minimum height=3.25cm},
	visnestmid/.style={visbox, minimum width=5.7cm, minimum height=2.2cm},
	visnestinn/.style={visbox, minimum width=3.4cm, minimum height=1.05cm,
		inner sep=3pt, font=\scriptsize\bfseries, align=center, text=methD},
	visnestpos/.style={anchor=south east},
	visnestinset/.style={xshift=-5pt, yshift=5pt},
}

\newcommand{\rank}{\text{rank}}

\newcommand{\Diag}{\text{Diag}}
\newcommand{\diag}{\text{diag}}

\DeclareMathOperator{\range}{range}
\DeclareMathOperator{\spanop}{span}
\newcommand{\Spsd}[1]{\mathbb{S}^{#1}_+}

\DeclareMathOperator{\msd}{MSD}
\DeclareMathOperator{\sd}{sd}

\newcommand{\R}{\mathbb{R}}

\newcommand{\C}{\mathcal{C}}

\newtheorem{prop}{Proposition}
\newtheorem{lem}{Lemma}[section]
\newtheorem{thm}{Theorem}[section]
\newtheorem{cor}{Corollary}[section]

\newtheorem{ex}{Example}[section]

\crefname{thm}{Theorem}{Theorems}
\Crefname{thm}{Theorem}{Theorems}
\crefname{problem}{Problem}{Theorems}
\Crefname{problem}{Problem}{Theorems}
\Crefname{assump}{Assumption}{Theorems}
\crefname{assump}{Assumption}{Theorems}
\crefname{assumption}{Assumption}{Assumptions}
\Crefname{assumption}{Assumption}{Assumptions}
\crefname{conjecture}{Conjecture}{Theorems}
\Crefname{conjecture}{Conjecture}{Theorems}
\crefname{prop}{Proposition}{Propositions}
\Crefname{prop}{Proposition}{Propositions}
\crefname{cor}{Corollary}{Corollaries}
\Crefname{cor}{Corollary}{Corollaries}
\crefname{lem}{Lemma}{Lemmas}
\Crefname{lem}{Lemma}{Lemmas}
\theoremstyle{definition}
\crefname{conj}{Conjecture}{Conjectures}
\Crefname{conj}{Conjecture}{Conjectures}
\crefname{remark}{Remark}{Remarks}
\Crefname{remark}{Remark}{Remarks}
\crefname{rmk}{Remark}{Remarks}
\Crefname{rmk}{Remark}{Remarks}
\crefname{example}{Example}{Examples}
\Crefname{example}{Example}{Examples}
\crefname{align}{}{}
\Crefname{align}{}{}
\crefname{equation}{}{}
\Crefname{equation}{}{}

\def\eqref#1{{\normalfont(\ref{#1})}}

\usepackage{authblk}

\author[1]{Hao Hu\footnote{School of Mathematical and Statistical Sciences, Clemson University, Clemson, USA; Email: \url{hhu2@clemson.edu}; Research supported by the Air Force Office of Scientific Research under award number FA9550-23-1-0508.}}

\begin{document}

\title{Sharp
Singularity-Degree Bounds for Equality-Generated SDP--RLT Relaxations of Binary Programs}

\break
\date{}
\maketitle

\medskip

\begin{abstract}
    Singularity degree is an important measure of semidefinite programming (SDP) degeneracy, but it is generally unavailable a priori from the problem data. We augment the Shor relaxation of nonempty binary sets \(\{x\in\{0,1\}^n:Ax=b\}\) with the first-level Reformulation--Linearization Technique (RLT) equations generated by the defining linear equalities. For the resulting equality-generated SDP--RLT relaxation, we determine the exact worst-case singularity degree. If \(\operatorname{rank}(A)=m\) and \(0<m<n\), then the associated relaxation has singularity degree at most \(\min\{m,n-m\}\), and this rank--nullity bound is attained for every possible rank in this range. Consequently, the worst-case singularity degree over this class is \(\lfloor n/2\rfloor\) for \(n\geq2\). This is strikingly smaller than the sharp general bound \(n\) for feasible SDP systems with matrix variables of order \(n+1\) \cite[Example~2]{sturm2000error}. Thus, for individual relaxations, rank and nullity provide an a priori bound on the otherwise inaccessible singularity degree and on the H\"older exponent in error bounds estimating distance to feasibility from constraint residuals.
\end{abstract}


\section{Introduction}

The singularity degree of a semidefinite program quantifies its failure of
Slater's condition: it is the minimum number of facial-reduction
steps needed to reach the minimal face containing the feasible set.  Although
it governs the worst-case H\"older behavior of error bounds and is closely
connected with the numerical sensitivity of degenerate SDPs
\cite{sturm2000error,sremac2021error}, it is generally an a posteriori
parameter.  Determining it generally requires analyzing the facial-reduction
certificates leading to the minimal face, and no tractable procedure is known
for general spectrahedra
\cite{sremac2017complete}.  The maximum singularity degree is the largest
possible length of a facial-reduction sequence, and computing it is NP-hard
for general SDPs \cite{hu2024maximum}.  Singularity
degree is therefore used primarily to explain pathological behavior, rather
than as a quantity that is readily computable from the problem data.

This makes sharp a priori bounds for important structured classes especially
valuable: even when the singularity degree of a particular instance is unknown,
a uniform bound for the class determines in advance its largest possible
value.  Through existing error-bound theory, such a bound supplies an a priori
bound on the singularity-degree exponent governing the H\"older relationship
between constraint residuals and distance to feasibility.  An important
structured class consists of equality-generated SDP relaxations based on the
Reformulation--Linearization Technique (RLT)
\cite{sherali1990hierarchy,sherali2013reformulation}.
These relaxations offer a practical modeling choice when the Shor relaxation
provides insufficient bounds but imposing the full collection of McCormick
inequalities is too costly.
We therefore seek a sharp upper bound on the singularity degree of these
relaxations.

To formalize this question, let
\[
    P(A,b):=\{x\in\{0,1\}^n\mid Ax=b\}
\]
be the binary set defined by the linear equalities \(Ax=b\).  We define its
equality-generated SDP--RLT relaxation by
\[
\mathcal R(A,b)
:=
\left\{
Y=\begin{pmatrix}1&x^T\\x&X\end{pmatrix}\succeq0
\ \middle|\
\diag(X)=x,\ Ax=b,\ AX=bx^T
\right\}.
\]
For a fixed number \(n\) of binary variables, we study the extremal problem
\begin{equation}\label{eq:intro-extremal-problem}
    \max_{A,b:\,P(A,b)\neq\emptyset}
    \sd\bigl(\mathcal R(A,b)\bigr).
\end{equation}
We call the value in \eqref{eq:intro-extremal-problem} the
\emph{worst-case singularity degree} of the equality-generated SDP--RLT class
in \(n\) binary variables.
The maximization ranges over affine descriptions \((A,b)\) satisfying
\(P(A,b)\neq\emptyset\), with each description inducing the relaxation
\(\mathcal R(A,b)\).  Different
descriptions of the same binary set may induce relaxations with different
singularity degrees.

Our main result gives the exact answer: the value in
\eqref{eq:intro-extremal-problem} is
\[
    \begin{cases}
        1,&n=1,\\[1mm]
        \lfloor n/2\rfloor,&n\geq2.
    \end{cases}
\]
The result has a sharper rank--nullity form.  Let \(m=\rank(A)\).  For
\(0<m<n\),
\[
    \sd\bigl(\mathcal R(A,b)\bigr)\leq \min\{m,n-m\},
\]
and this bound is attained for every possible rank \(m\) in this range.
Therefore, \(\min\{m,n-m\}\) is the best possible uniform bound based only
on rank and nullity.  Maximizing it over \(m\) gives
\(\lfloor n/2\rfloor\), which occurs when rank and nullity are as nearly
balanced as possible.  For \(n\geq2\), this worst-case value can be attained by
a homogeneous system with \(P(A,0)=\{0\}\).  Thus even a singleton binary set
can yield a lifted relaxation with singularity degree \(\lfloor n/2\rfloor\).

The proof homogenizes the affine equations and uses a nullspace
parametrization to obtain a reduced formulation whose matrix variable has
order \(r+1\), where \(r:=n-m\) is the nullity of \(A\).  For this formulation,
we prove the stronger statement that its maximum singularity degree is at most
\(\min\{m,r\}\).  Consequently, every facial-reduction sequence of the
reduced formulation has length at most \(\min\{m,r\}\), independently of the
choice of exposing matrices.  This stronger bound for the reduced formulation
is the key mechanism behind the singularity-degree bound for the original
relaxation \(\mathcal R(A,b)\).

\paragraph{Relation to prior work.}
Facial reduction was introduced by Borwein and Wolkowicz
\cite{borwein1981regularizing,borwein1981facial}, and Sturm introduced
singularity degree and established its connection with error bounds
\cite{sturm2000error}.  Later developments include Pataki's treatment of
facial reduction and extended duality \cite{pataki2013strong}, as well as the
iteration bounds of Louren\c{c}o, Muramatsu, and Tsuchiya based on partial
polyhedrality \cite{lourencco2018facial}.  Facial reduction for structured
conic and polynomial optimization problems is studied in
\cite{waki2013facial,drusvyatskiy2017many,sremac2021error}.  For SDP
relaxations of nonconvex sets, Tun\c{c}el relates the existence of Slater points
to the dimension and affine hull of the original set
\cite{tuncel2001slater}.

Several results derive singularity-degree bounds from problem structure.  For
positive semidefinite matrix completion, graph structure controls the degree
\cite{tanigawa2017singularity}.  For SDP relaxations of binary sets defined by
linear equalities, the basic Shor relaxation has singularity degree at most one
\cite{hu2026shorsingularity}.  By contrast, general order-\((n+1)\) SDP systems
can attain singularity degree \(n\) \cite[Example~2]{sturm2000error}.  Our
theorem gives an analogous structure-dependent bound for equality-generated
SDP--RLT relaxations and places their worst-case singularity degree between
these two benchmarks.

The relaxations studied here augment the Shor relaxation with the standard
first-level RLT equations generated by linear equalities
\cite{sherali1990hierarchy,sherali2013reformulation}.
Recent work has also used facial reduction to regularize
SDP--RLT relaxations with linear equality constraints
\cite{yildirim2026relaxations} and has studied \(AX=0\) facial constraints
arising from equality-generated RLT and moment--SOS relaxations
\cite{hou2026lowrankalm}.  These works do not determine the singularity degree
of the resulting formulations; the present paper establishes its sharp
class-wide rank--nullity bound in the binary setting.
Sturm's example, by contrast, does not arise from an SDP relaxation of a
binary program and therefore does not resolve the present question.
The sharpness construction developed in this paper yields a
family of equality-generated first-level SDP--RLT relaxations of linearly
constrained binary sets whose singularity degree grows linearly with the number
of variables.

\paragraph{Organization.}
\Cref{sec_prel} introduces notation and the facial-reduction terminology used
throughout.  \Cref{sec:general_equality_generated} defines
the equality-generated SDP--RLT relaxation, derives its reduced formulation,
and relates the facial-reduction sequences of the two systems.
\Cref{sec:binary_high_singularity_degree} develops the Vandermonde--Hankel
construction that attains the worst-case value.
\Cref{sec:structured_upper_bound} proves the sharp rank--nullity upper
bound, resolves \eqref{eq:intro-extremal-problem}, and establishes attainment
for every rank--nullity pair.  The appendix proves a linear-span identity used
to compare the original and reduced formulations.

\section{SDP relaxations of binary programs and facial reduction}
\label{sec_prel}

\subsection{Notation}

For a positive integer \(p\), let \(\mathbb R^p\) be Euclidean space,
\(\mathbb S^p\) the space of real symmetric \(p\times p\) matrices, and
\(\mathbb S^p_+\) and \(\mathbb S^p_{++}\) the positive semidefinite and
positive definite cones, respectively.  We use
\(\langle X,Y\rangle=\operatorname{tr}(XY)\) on symmetric matrix spaces.
We label the coordinates of \(\mathbb R^p\) by \(1,\ldots,p\) and those of its
homogenization \(\mathbb R^{p+1}\) by \(0,1,\ldots,p\).  When the ambient space
is clear, \(e_i\) denotes the standard unit vector associated with coordinate
\(i\); in the homogenized space, \(e_0\) corresponds to the homogenizing
coordinate, while \(e_1,\ldots,e_p\) correspond to the original coordinates.
We use the symmetric basis
\begin{equation}\label{eq:symmetric_basis_matrices}
E_{ij}:=\frac12(e_ie_j^T+e_je_i^T)\quad(i\neq j),
\qquad
E_{ii}:=e_ie_i^T .
\end{equation}
For a square matrix \(X\), \(\diag(X)\) is its diagonal vector.  For a matrix
or linear map \(T\), \(\range(T)\) and \(\ker(T)\) denote its range and
kernel, and \(T^*\) denotes the adjoint.  For a cone \(K\) in an inner-product
space, \(K^*:=\{z\mid \langle z,x\rangle\geq0\text{ for every }x\in K\}\)
denotes its dual cone.  For a set \(S\) in an inner-product space,
\[
S^\perp:=\{z\mid \langle z,s\rangle=0\text{ for every }s\in S\}
       =\operatorname{span}(S)^\perp .
\]

\subsection{Facial reduction and singularity degree}
\label{subsec:facial_reduction_algorithm}

Let \(L\) be an affine subspace with
\(L\cap\mathbb S^p_+\neq\emptyset\).  Slater's condition is
\(L\cap\mathbb S^p_{++}\neq\emptyset\).  If it fails, a theorem of the
alternative yields a nonzero matrix
\(W\in L^\perp\cap\mathbb S^p_+\).  Its
orthogonal hyperplane exposes a proper face of the current PSD cone that still
contains \(L\cap\mathbb S^p_+\)
\cite{borwein1981regularizing,borwein1981facial}.

Every face of \(\mathbb S^p_+\) can be written as
\[
F
=
\{Y\succeq0\mid \range(Y)\subseteq\mathcal V\}
=
\{VRV^T\mid R\in\mathbb S^r_+\},
\]
where the columns of \(V\in\mathbb R^{p\times r}\) span \(\mathcal V\).
Let \(F\) be a current face containing
\(L\cap\mathbb S^p_+\).  A matrix
\(W\in L^\perp\cap(F^*\setminus F^\perp)\) is an exposing matrix for a
facial-reduction step on \(F\).
The condition \(W\in L^\perp\) ensures that the new face
still contains the feasible set.  Moreover, \(W\in F^*\) gives
\(V^TWV\succeq0\), while \(W\notin F^\perp\) makes this matrix nonzero.
Hence \(F\cap W^\perp\) is a proper face of \(F\).  A
\emph{facial-reduction sequence of length \(d\)}
is constructed as follows.  Set
\(F_0:=\mathbb S^p_+\).  For \(j=1,\ldots,d\), the \(j\)th facial-reduction
step selects an exposing matrix
\[
W_j\in L^\perp\cap(F_{j-1}^*\setminus F_{j-1}^\perp)
\]
and set \(F_j:=F_{j-1}\cap W_j^\perp\).  The terminal condition is
\(F_d=F_{\min}\), where \(F_{\min}\) is the minimal face containing the
feasible set.  Consequently,
\[
\mathbb S^p_+=F_0\supsetneq F_1\supsetneq\cdots
\supsetneq F_d=F_{\min}.
\]
Thus every step properly reduces the current face.  A
\emph{partial facial-reduction sequence}
satisfies the same step conditions but is not required to satisfy the terminal
condition \(F_d=F_{\min}\).

More generally, if \(F\) is a face of \(\mathbb S^p_+\) and
\(L\cap F\neq\emptyset\), a facial-reduction sequence for \(L\cap F\) is
defined by taking \(F_0:=F\) and applying the same step condition.  Its terminal
face is the minimal face of \(F\) containing \(L\cap F\).

The minimum possible length \(d\) of a facial-reduction sequence is the
singularity degree, denoted by
\(\sd(L\cap\mathbb S^p_+)\).
The maximum possible length of a facial-reduction sequence is the
\emph{maximum singularity degree}, denoted by
\(\msd(L\cap\mathbb S^p_+)\) \cite[Section~2.1]{hu2024maximum}.  Therefore,
\begin{equation}\label{eq:sd-le-msd}
    \sd(L\cap\mathbb S^p_+)
    \leq
    \msd(L\cap\mathbb S^p_+).
\end{equation}

\section{Equality-generated SDP--RLT relaxations and facial-reduction tools}
\label{sec:general_equality_generated}

\subsection{Definition and omitted redundant equations}
\label{sec:equality-generated-rlt-relaxation}

Let \(A\in\R^{q\times n}\) and let \(b\in\R^q\).  Recall
that the associated linearly constrained binary set is
\begin{equation}\label{eq:PAb}
    P:=P(A,b) = \{x\in\{0,1\}^n\mid Ax=b\}.
\end{equation}
Its equality-generated SDP--RLT relaxation is
\begin{equation}\label{eq:RAb}
\mathcal R(A,b)
:=
\left\{
Y=\begin{pmatrix}1&x^T\\x&X\end{pmatrix}\in\Spsd{n+1}
\ \middle|\
\diag(X)=x,\quad
\begin{pmatrix}-b&A\end{pmatrix}Y=0
\right\}.
\end{equation}
We index the rows and columns of \(Y\) by \(0,1,\ldots,n\), where index
\(0\) corresponds to the constant monomial.  Thus, \(Y_{00}=1\) is the
normalization equation, and \(\diag(X)=x\) is equivalently the family of
binary arrow equations \(Y_{ii}=Y_{0i}\), \(i=1,\ldots,n\).
The matrix equation in \eqref{eq:RAb} is equivalent to \(Ax=b\) and
\(AX=bx^T\).  These equations have the standard degree-one RLT interpretation.
Indeed, for every row
\(a_i^T\) of \(A\) and every \(j=1,\ldots,n\), multiplying
\(a_i^Tx=b_i\) by \(x_j\) and linearizing gives
\[
    (AX)_{ij}=b_i x_j.
\]
Throughout the remainder of the paper, we assume that \(P\neq\emptyset\).
Every \(x\in P\) yields the feasible rank-one matrix
\((1,x^T)^T(1,x^T)\in\mathcal R(A,b)\), so \(\mathcal R(A,b)\neq\emptyset\).
The same assumption supplies the binary feasible point used in the reduction
to homogeneous equations in the next subsection.

The following operations leave both the SDP relaxation and its singularity and
maximum singularity degrees unchanged because they do not change the span of
the homogeneous constraint matrices:
\begin{itemize}[leftmargin=2em]
\item
including the RLT equations generated using \(1-x_j\);
\item
including the linearized products between pairs of rows of \(Ax=b\);
\item
removing redundant rows of \(Ax=b\).
\end{itemize}
We therefore retain none of these redundant equations and assume that
\(A\in\R^{m\times n}\) has full row rank, where \(m=\rank(A)\).  By contrast,
the McCormick inequalities in the full first-level SDP--RLT relaxation
generally strengthen the relaxation and, when represented using nonnegative
slacks, produce a conic formulation over the product of the positive
semidefinite cone and a nonnegative orthant.  Their singularity degree requires
a separate analysis.

\subsection{Reduction of affine equations to homogeneous form}

We next apply the standard \emph{switching operation}, also called
binary-variable complementation, to move a binary feasible point to the
origin without leaving the binary cube; see, e.g., \cite{ziegler2000zeroone}.
Fix \(\overline x\in P\).  We switch precisely the variables for which
\(\overline x_i=1\): if \(\overline x_i=0\), set
\(w_i=x_i\), whereas if \(\overline x_i=1\), set
\(w_i=1-x_i\).  Equivalently, set
\[
    D:=\Diag(\mathbf1-2\overline x),\qquad
    w=D(x-\overline x),\qquad
    x=\overline x+Dw.
\]
Thus \(x\in\{0,1\}^n\) if and only if \(w\in\{0,1\}^n\), and
\(\overline x\) is mapped to \(w=0\).  Define
\[
    \widehat A:=AD,
    \qquad
    \widehat P:=\{w\in\{0,1\}^n\mid \widehat Aw=0\}.
\]
Since \(A\overline x=b\),
\[
    Ax=b
    \quad\Longleftrightarrow\quad
    \widehat Aw=0.
\]
Hence the change of variables maps \(P\) bijectively onto \(\widehat P\),
whose equality-generated SDP--RLT relaxation is
\(\mathcal R(\widehat A,0)\).
The next lemma shows that this change of variables extends to an automorphism
of the PSD cone that maps the affine subspace defined by the equality
constraints of \(\mathcal R(\widehat A,0)\) onto that defined by the equality
constraints of \(\mathcal R(A,b)\).
Consequently, it preserves facial-reduction sequences and singularity degree.

\begin{lem}\label{lem:affine-to-homogeneous}
With the notation above,
\[
    \sd(\mathcal R(A,b))=\sd(\mathcal R(\widehat A,0)).
\]
\end{lem}

\begin{proof}
Set
\[
    Q:=
    \begin{pmatrix}
        1&0\\
        \overline x&D
    \end{pmatrix},
    \qquad
    \Psi(\widehat Y):=Q\widehat YQ^T.
\]
The diagonal entries of \(D\) belong to \(\{-1,1\}\), so \(Q\) is
invertible.  Hence \(\Psi\) is an invertible linear map.

Let \(\widehat L\) and \(L\) denote the affine subspaces defined by the equality
constraints of \(\mathcal R(\widehat A,0)\) and \(\mathcal R(A,b)\),
respectively.  We show that \(\Psi(\widehat L)=L\).  For
\(Y:=\Psi(\widehat Y)\) and every \(H\in\mathbb S^{n+1}\),
\begin{equation}\label{eq:switching-adjoint-identity}
    \langle H,Y\rangle
    =\langle Q^THQ,\widehat Y\rangle.
\end{equation}
Thus the normalization and arrow equations can be compared by transforming
their coefficient matrices by \(H\mapsto Q^THQ\).  Because
\(\overline x_i\in\{0,1\}\) and
\(D_{ii}=1-2\overline x_i\), direct calculation gives
\[
    Q^TE_{00}Q=E_{00},
    \qquad
    Q^T(E_{ii}-E_{0i})Q=E_{ii}-E_{0i}
    \quad (i=1,\ldots,n).
\]
Thus the normalization equation and the binary arrow equations, including
their right-hand sides, correspond exactly under \(\Psi\).

Let \(B:=[-b\ \ A]\) and \(\widehat B:=[0\ \ \widehat A]\).  Then
\[
    BQ
    =
    \begin{pmatrix}-b+A\overline x&AD\end{pmatrix}
    =
    \widehat B.
\]
Consequently, \(BY=BQ\widehat YQ^T=\widehat B\widehat YQ^T\).  Since
\(Q^T\) is invertible, \(BY=0\) if and only if
\(\widehat B\widehat Y=0\).
Together with the identities for the normalization and arrow equations, this
proves \(\Psi(\widehat L)=L\).

Since \(\Psi\) maps the PSD cone onto itself and \(\Psi(\widehat L)=L\), the
standard invariance of facial reduction under cone automorphisms gives a
length-preserving bijection between the facial-reduction sequences of the two
formulations.  They therefore have equal singularity degree.
\end{proof}

The proof also shows that
\(\mathcal R(A,b)=\Psi(\mathcal R(\widehat A,0))\).  Hence every affine system
with a binary feasible point can be reduced to a homogeneous system by
complementing binary variables.  We therefore work below with \(b=0\) and
reuse \(A\) for the resulting homogeneous constraint matrix.

\subsection{Reduction to the facially reduced formulation}

This subsection shows that one facial-reduction step for \(\mathcal R(A,0)\)
exposes the face forced by \(Ax=0\) and \(AX=0\).  Parametrizing this face
produces the smaller system \(\mathcal T(C)\), whose constraints are precisely
the remaining normalization and arrow equations.  We first establish the
exact relation between the two formulations and then compare their
singularity degrees.  This explains why the subsequent analysis may focus on
\(\mathcal T(C)\).

Assume in this subsection that \(m\geq1\), so \(A\neq0\).  If \(m=0\), there
are no linear equations and \(\mathcal R(0,0)\) is strictly feasible, so its
singularity degree is zero.  Indeed, take \(x=\frac12\mathbf1\) and
\(X=\frac14\mathbf1\mathbf1^T+\frac14I\).  Then \(\diag(X)=x\), and the
Schur complement is \(X-xx^T=\frac14I\succ0\).

Let \(A\in\R^{m\times n}\) have full row rank and set
\[
    r:=n-m.
\]
Let \(C\in\R^{n\times r}\) have full column rank with
\[
    \ker A=\range(C).
\]
Thus \(m=\rank A\) is the number of independent equations in \(Ax=0\),
\(r=\dim\ker A\) is the reduced dimension, and \(n=m+r\).  We abbreviate
\begin{equation}\label{eq:RA}
\mathcal R(A):=\mathcal R(A,0)
=
\left\{
\begin{pmatrix}1&x^T\\x&X\end{pmatrix}\in\Spsd{n+1}
\ \middle|\
\diag(X)=x,\quad Ax=0,\quad AX=0
\right\}.
\end{equation}
We introduce the following notation for an arbitrary matrix \(H\).
For any integer \(q\geq0\) and any \(h\in\R^q\), define
\[
\Phi(h)
:=
\begin{pmatrix}
0&-\frac12h^T\\[1mm]
-\frac12h&hh^T
\end{pmatrix}
\]
and, for any matrix \(H\in\R^{n\times q}\) with rows \(h_i^T\), set
\begin{equation}\label{eq:T-general-system}
\mathcal T(H)
:=
\left\{
Z\in\Spsd{q+1}
\ \middle|\
Z_{00}=1,\quad
\langle\Phi(h_i),Z\rangle=0,\quad i=1,\ldots,n
\right\}.
\end{equation}
Writing the rows of \(C\) as \(c_i^T\), the choice \(H=C\) defines the
candidate reduced system \(\mathcal T(C)\).  To identify it with the
formulation obtained after the facial-reduction step constructed below, set
\begin{equation}\label{eq:equality-face-matrices}
    V:=\begin{pmatrix}1&0\\0&C\end{pmatrix},
    \qquad
    \widetilde A:=\begin{pmatrix}0&A\end{pmatrix},
    \qquad
    W:=\widetilde A^T\widetilde A
       =\begin{pmatrix}0&0\\0&A^TA\end{pmatrix}.
\end{equation}

\begin{lem}\label{lem:equality-face-reduction}
The matrix \(W\) is an exposing matrix for a first facial-reduction step of
\(\mathcal R(A)\).  The face exposed by this step is
\[
    F_A
    :=\Spsd{n+1}\cap W^\perp
    =\{VZV^T:Z\in\Spsd{r+1}\}.
\]
Under the parametrization \(Z\mapsto VZV^T\), restricting the defining SDP
system to \(F_A\) yields \(\mathcal T(C)\).
In particular,
\begin{equation}\label{eq:RA-reduced-identity}
    \mathcal R(A)
    =
    \{VZV^T:Z\in\mathcal T(C)\}.
\end{equation}
\end{lem}

\begin{proof}
The equations \(Ax=0\) and \(AX=0\) are precisely
\[
    Y\widetilde A^T=0.
\]
The linear span of their constraint matrices contains
\(W=\widetilde A^T\widetilde A\).  Since \(W\succeq0\) and \(W\neq0\), it is
an exposing matrix for a first facial-reduction step of \(\mathcal R(A)\).

It remains to identify the exposed face.  Since \(\ker A=\range(C)\),
\[
    \ker(\widetilde A)=\range(V)=\ker(W).
\]
The characterization of faces of \(\Spsd{n+1}\) therefore gives
\[
    \Spsd{n+1}\cap W^\perp
    =\{VZV^T:Z\in\Spsd{r+1}\}
    =F_A.
\]
Hence every \(Y\in\mathcal R(A)\) can be written as \(Y=VZV^T\) for some
\(Z\succeq0\).

The preceding argument proves the assertions about \(W\) and \(F_A\).  It
remains to identify the restricted SDP system as \(\mathcal T(C)\).  Since
\(\widetilde A V=0\), every constraint matrix \(H\) associated with \(Ax=0\)
and \(AX=0\) satisfies \(V^THV=0\).  It therefore suffices to compare the
normalization and arrow equations under this parametrization.
Using the symmetric basis from \eqref{eq:symmetric_basis_matrices}, a direct
calculation gives
\[
    V^T(E_{ii}-E_{0i})V
    =\Phi(c_i).
\]
Hence, for \(Y=VZV^T\),
\[
    Y_{ii}-Y_{0i}
    =\langle E_{ii}-E_{0i},Y\rangle
    =\langle\Phi(c_i),Z\rangle.
\]
Likewise, the coefficient matrix \(E_{00}\in\mathbb S^{n+1}\) of
\(Y_{00}=1\) maps to \(E_{00}\in\mathbb S^{r+1}\), the coefficient matrix of
\(Z_{00}=1\).  Thus the restricted defining equations are exactly those of
\(\mathcal T(C)\), which in particular proves
\eqref{eq:RA-reduced-identity} and completes the proof.
\end{proof}

We next record the general mechanism that allows two facial-reduction steps to
be combined in the present setting.

\begin{lem}\label{lem:merge-fr-steps}
Let \(L\subseteq\mathbb S^p\) be an affine subspace such that
\(L\cap\mathbb S^p_+\) is nonempty.  Suppose that
\(W_1\in L^\perp\cap\mathbb S^p_+\) exposes a proper face
\(F:=\mathbb S^p_+\cap W_1^\perp\) containing
\(L\cap\mathbb S^p_+\), and assume that
\begin{equation}\label{eq:face-orthogonal-contained}
    F^\perp\subseteq L^\perp.
\end{equation}
Then, for every \(W_2\in L^\perp\cap(F^*\setminus F^\perp)\), there exists
\(\overline W\in L^\perp\cap\mathbb S^p_+\) such that
\[
    \mathbb S^p_+\cap\overline W^\perp
    =F\cap W_2^\perp.
\]
\end{lem}

\begin{proof}
For every face \(F\) of the positive semidefinite cone, the standard identity
\(F^*=\mathbb S^p_++F^\perp\) holds.  Hence we may write
\(W_2=\widetilde W_2+N\), where \(\widetilde W_2\succeq0\) and
\(N\in F^\perp\).
By \eqref{eq:face-orthogonal-contained}, \(N\in L^\perp\).  Since also
\(W_2\in L^\perp\), it follows that \(\widetilde W_2\in L^\perp\).
Set \(\overline W:=W_1+\widetilde W_2\), which belongs to
\(L^\perp\cap\mathbb S^p_+\).  Since both summands are positive semidefinite,
and since \(N\in F^\perp\) implies that \(W_2\) and \(\widetilde W_2\) agree
on \(F\), we obtain
\[
    \mathbb S^p_+\cap\overline W^\perp
    =F\cap\widetilde W_2^\perp
    =F\cap W_2^\perp.
\]
This proves the result.
\end{proof}

The condition \eqref{eq:face-orthogonal-contained} is equivalent to
\(L\subseteq\operatorname{span}(F)\).  Thus the entire affine subspace,
rather than only its positive semidefinite part, already lies in the linear
span of the exposed face.  In the application below, the equations \(Ax=0\)
and \(AX=0\) enforce this property for \(F_A\).  Consequently, the first
facial-reduction step replaces the ambient cone by a face whose linear span is
already imposed by the affine equations.  This structural property is what
allows the first step to be combined with a subsequent facial-reduction step.

The next proposition compares the lengths of facial-reduction sequences for
\(\mathcal R(A)\) and its reduced formulation \(\mathcal T(C)\).  We state
the stronger bound on maximum singularity degree because the later upper-bound
argument in
Section~\ref{sec:structured_upper_bound} must control arbitrary
facial-reduction sequences, not only shortest ones.

\begin{prop}\label{prop:reduced-order-bound}
Assume \(r\geq1\).  Then
\[
    \msd(\mathcal T(C))\leq r
    \qquad\text{and}\qquad
    \sd(\mathcal R(A))
    =
    \max\{1,\sd(\mathcal T(C))\}
    \leq r.
\]
\end{prop}

\begin{proof}
Since \(E_{00}\in\mathcal T(C)\), the standard matrix-order bound for a
nonzero feasible system over \(\mathbb S^{r+1}_+\) gives
\(\msd(\mathcal T(C))\leq r\).

We next prove the formula for \(\sd(\mathcal R(A))\).  Let \(L\) be the affine
subspace defined by the constraints in \eqref{eq:RA}.  By
\Cref{lem:equality-face-reduction}, the matrix \(W\) in
\eqref{eq:equality-face-matrices} belongs to
\(L^\perp\cap\mathbb S^{n+1}_+\) and exposes \(F_A\).
Let \(L_A\) denote the linear span of the constraint matrices for
\(Ax=0\) and \(AX=0\).  A direct linear-algebra calculation gives
\begin{equation}\label{eq:general_equality_constraint_span}
    L_A
    =
    \{N\in\mathbb S^{n+1}:V^TNV=0\}
    =F_A^\perp;
\end{equation}
see Appendix~\ref{app:rlt_equality_constraint_span}.  Every coefficient matrix
of \(Ax=0\) and \(AX=0\) is orthogonal to every \(Y\in L\), because these
constraints have zero right-hand sides.  Hence \(L_A\subseteq L^\perp\), and
\(F_A^\perp\subseteq L^\perp\), which is precisely the assumption in
\Cref{lem:merge-fr-steps}.

Let \(d:=\sd(\mathcal T(C))\).  By
\Cref{lem:equality-face-reduction}, after the initial step exposing \(F_A\),
the resulting formulation is \(\mathcal T(C)\).  Hence a shortest
facial-reduction sequence for \(\mathcal T(C)\) supplies \(d\) further steps.
If \(d=0\), then \(F_A\) is the minimal face containing \(\mathcal R(A)\), and
the step exposed by \(W\) gives \(\sd(\mathcal R(A))=1\).  If \(d\geq1\),
\Cref{lem:merge-fr-steps} combines the exposure by \(W\) with the first of
these \(d\) additional steps.  Applying the remaining \(d-1\) steps gives
\[
    \sd(\mathcal R(A))\leq d.
\]
Conversely, intersect the faces in a shortest facial-reduction sequence for
\(\mathcal R(A)\) with \(F_A\).  After repeated faces are discarded, the
remaining faces form a facial-reduction sequence for the formulation on
\(F_A\), of length at most \(\sd(\mathcal R(A))\).  Under
\(Z\mapsto VZV^T\), this becomes a facial-reduction sequence for
\(\mathcal T(C)\).  Therefore,
\[
    d\leq\sd(\mathcal R(A)).
\]
Combining the two cases yields
\[
    \sd(\mathcal R(A))
    =
    \max\{1,\sd(\mathcal T(C))\}.
\]
Finally, \eqref{eq:sd-le-msd} gives
\(\sd(\mathcal T(C))\leq\msd(\mathcal T(C))\leq r\).  Since \(r\geq1\), the
displayed maximum is at most \(r\).
\end{proof}

\subsection{Facial reduction under restriction to a face}
\label{subsec:restriction-psd-face}

The preceding proof used the fact that, after each face in a facial-reduction
sequence is intersected with a fixed face, every strict inclusion that remains
is a valid facial-reduction step for the restricted problem.  We now record
the general form of this restriction principle for later use.  Here the face
need not contain every feasible matrix, and the restriction is simply an
intersection in the original matrix space.

\begin{lem}[Restriction of a facial-reduction sequence]
\label{lem:restrict-fr-sequence}
Let \(L\subseteq\mathbb S^p\) be affine, let \(F\) be a face of
\(\Spsd{p}\) such that \(L\cap F\neq\emptyset\), and let
\(F_0,\ldots,F_d\) be a facial-reduction sequence for
\(L\cap\Spsd{p}\).  For \(j=1,\ldots,d\), let \(W_j\) be the exposing matrix
used in the step from \(F_{j-1}\) to \(F_j\).  For \(j=0,\ldots,d\), set
\[
    \overline F_j:=F_j\cap F.
\]
Then each \(\overline F_j\) is a face of \(F\) containing
\(L\cap F\), and
\[
    \overline F_j
    =
    \overline F_{j-1}\cap W_j^\perp
    \qquad (j=1,\ldots,d).
\]
After repeated faces are omitted from
\(\overline F_0,\ldots,\overline F_d\), the remaining faces and corresponding
matrices \(W_j\) form a partial facial-reduction sequence for \(L\cap F\).
\end{lem}

\begin{proof}
Since \(F_j\) and \(F\) are faces of \(\Spsd{p}\), their intersection
\(\overline F_j\) is a face of \(F\) containing \(L\cap F\).  Moreover,
\[
    \overline F_j
    =
    (F_{j-1}\cap W_j^\perp)\cap F
    =
    \overline F_{j-1}\cap W_j^\perp.
\]
For each retained strict inclusion, the original sequence
gives \(W_j\in L^\perp\cap F_{j-1}^*\).  Since
\(\overline F_{j-1}\subseteq F_{j-1}\), we have
\(W_j\in\overline F_{j-1}^*\), while strictness gives
\(W_j\notin\overline F_{j-1}^\perp\).  Hence
\(W_j\in L^\perp\cap
(\overline F_{j-1}^*\setminus\overline F_{j-1}^\perp)\).  Thus \(W_j\)
exposes \(\overline F_j\) from \(\overline F_{j-1}\).
Removing the indices \(j\) for which
\(\overline F_j=\overline F_{j-1}\), together with the corresponding matrices
\(W_j\), yields the asserted partial facial-reduction sequence.
\end{proof}

The restriction principle also gives the following bound on
maximum singularity degree.

\begin{lem}[Maximum singularity degree under restriction]
\label{lem:msd-restriction-bound}
Let \(L\subseteq\mathbb S^p\) be affine.  Let
\(\mathcal G\subseteq\mathbb R^p\) be a subspace of dimension \(q\), and let
\(F_{\mathcal G}\) be the face of \(\mathbb S^p_+\) associated with
\(\mathcal G\).  Assume that \(L\cap F_{\mathcal G}\neq\emptyset\), choose
\(U\in\mathbb R^{p\times q}\) with orthonormal columns spanning
\(\mathcal G\), and define
\[
    L_{\mathcal G}:=
    \{Z\in\mathbb S^q:UZU^T\in L\}.
\]
Then
\[
    \msd(L\cap\mathbb S^p_+)
    \leq
    \dim(\mathcal G^\perp)
    +\msd(L_{\mathcal G}\cap\mathbb S^q_+).
\]
\end{lem}

\begin{proof}
Consider an arbitrary facial-reduction sequence of length
\(d\) for \(L\cap\mathbb S^p_+\), with associated faces
\(F_0,\ldots,F_d\).  Let
\(\mathbb R^p=\mathcal V_0\supsetneq\cdots\supsetneq\mathcal V_d\) be their
associated subspaces, and set
\(\overline F_j:=F_j\cap F_{\mathcal G}\).  By
Lemma~\ref{lem:restrict-fr-sequence}, the strict inclusions among the faces
\(\overline F_j\) form a partial facial-reduction sequence for
\(L\cap F_{\mathcal G}\).  Under the face parametrization
\(Z\mapsto UZU^T\), this is a partial facial-reduction sequence for
\(L_{\mathcal G}\cap\mathbb S^q_+\).  This partial sequence can be extended
to a facial-reduction sequence.  Hence at most
\(\msd(L_{\mathcal G}\cap\mathbb S^q_+)\) indices satisfy
\(\overline F_j\subsetneq\overline F_{j-1}\).

It remains to count the indices for which
\(\overline F_j=\overline F_{j-1}\).  Set
\(\mathcal G_j:=\mathcal V_j\cap\mathcal G\), let
\(P_{\mathcal G^\perp}\) denote the orthogonal projector onto
\(\mathcal G^\perp\), and define
\[
    \delta_j:=
    \dim\bigl(P_{\mathcal G^\perp}(\mathcal V_j)\bigr).
\]
The face \(\overline F_j\) is associated with
\(\mathcal G_j\).  Moreover, the kernel of \(P_{\mathcal G^\perp}\)
restricted to \(\mathcal V_j\) is \(\mathcal G_j\), so rank--nullity gives
\[
    \dim\mathcal V_j=\dim\mathcal G_j+\delta_j.
\]
If \(\overline F_j=\overline F_{j-1}\), then
\(\mathcal G_j=\mathcal G_{j-1}\), whereas
\(\mathcal V_j\subsetneq\mathcal V_{j-1}\).  Therefore,
\(\delta_j\leq\delta_{j-1}-1\).  Since
\(\delta_0=\dim(\mathcal G^\perp)\), at most
\(\dim(\mathcal G^\perp)\) indices yield repeated faces.  Combining the two
counts proves the result.
\end{proof}

\section{A high-singularity-degree equality-generated SDP--RLT relaxation}
\label{sec:binary_high_singularity_degree}

Semidefinite programs can be highly ill-conditioned when Slater's condition
fails, and high singularity degree is a structural source of this difficulty
\cite{sturm2000error,sremac2021error}.  Although general SDP systems with high
singularity degree are known, those examples do not establish whether the same
behavior can occur in SDP relaxations of linearly constrained binary programs.
Such relaxations contain highly structured equations tied to individual
variables, such as the arrow constraints $Y_{ii}=Y_{0i}$ in
\eqref{eq:RAb}, which might appear to preclude long
facial-reduction sequences.  In particular, Sturm's classical example
\cite[Example~2]{sturm2000error} does not directly answer this question,
because it does not include the binary arrow constraints and its data matrices
arise from linearizing the quadratic equations $x_i^2=x_{i-1}$, rather than
from linear constraints on binary variables.

Here we answer the question positively.  For every $n\geq2$, we construct an
SDP relaxation of a linearly constrained binary set in $n$ variables whose
singularity degree is exactly $\lfloor n/2\rfloor$.  Thus, there exist such
relaxations whose singularity degree grows linearly with the number of binary
variables.

The proof proceeds in three stages.  We first construct a linearly constrained
binary singleton whose nullspace has a Vandermonde basis.  We then characterize
the positive semidefinite matrices in the linear span of its reduced
arrow-constraint matrices.  Finally, we combine these ingredients with the
reduction in Section~\ref{sec:general_equality_generated} to compute the exact
singularity degree.

Throughout this section, fix $n\geq2$ and set
\[
m:=\left\lceil\frac n2\right\rceil,
\qquad
r:=n-m=\left\lfloor\frac n2\right\rfloor.
\]
For $s=1,\ldots,r$, define the column vector
$c_s(t)\in\mathbb{R}^s$ by
\[
c_s(t):=\bigl(t,t^2,\ldots,t^s\bigr)^T.
\]
\subsection{A Vandermonde construction of a binary singleton}

Define $C\in\mathbb{R}^{n\times r}$ by
\begin{equation}\label{eq:rlt_linear_binary_C}
C_{kj}:=k^j,\qquad
k=1,\ldots,n,\quad j=1,\ldots,r.
\end{equation}
Thus, the $k$-th row of $C$ is
\(c_r(k)^T\).
The matrix $C$ has full column rank.  To give an explicit linear description
of its range, partition it as
\[
C=\begin{pmatrix}C_{\rm top}\\C_{\rm bot}\end{pmatrix},
\qquad C_{\rm top}\in\mathbb{R}^{r\times r},
\]
where $C_{\rm top}$ consists of the first $r$ rows.  The matrix
$C_{\rm top}$ is a row-scaled ordinary Vandermonde matrix and is nonsingular.
Set
\begin{equation}\label{eq:rlt_linear_binary_A}
A:=\begin{pmatrix}-C_{\rm bot}C_{\rm top}^{-1}&I_m\end{pmatrix}
\in\mathbb{R}^{m\times n}.
\end{equation}
Then $A$ has full row rank and
\begin{equation}\label{eq:rlt_kernel_A_range_C}
\ker A=\range(C).
\end{equation}

Consider the linearly constrained binary set
\begin{equation}\label{eq:rlt_linear_binary_set}
P_n
:=
\{x\in\{0,1\}^n\mid Ax=0\}.
\end{equation}
Thus $P_n$ is defined using only linear equalities and binarity.
Since the equations are homogeneous, \(0\in P_n\), so \(P_n\) is nonempty.
We prove the stronger fact that \(P_n\) is a singleton.  Thus the high
singularity degree established below occurs even though the underlying binary
feasible set is trivial.

\begin{prop}\label{prop:rlt_linear_binary_set_singleton}
The binary set $P_n$ defined in
\eqref{eq:rlt_linear_binary_set} is the singleton $\{0\}$.
\end{prop}

\begin{proof}
Let $x\in P_n$.  By \eqref{eq:rlt_kernel_A_range_C}, there is a
unique $y\in\mathbb{R}^r$ such that $x=Cy$.  Define the polynomial
\[
p(t):=c_r(t)^Ty.
\]
It has degree at most $r$ and satisfies $p(0)=0$.  Since
\(p(k)=x_k\in\{0,1\}\),
\[
p(k)\bigl(p(k)-1\bigr)=0,
\qquad k=1,\ldots,n.
\]
Thus the polynomial $p(t)(p(t)-1)$, whose degree is at most $2r$, vanishes
at the $n$ distinct positive evaluation points and also at $t=0$.  Since
$n+1>2r$, it must vanish identically.  Therefore $p$ is identically $0$ or
$1$.  The equality $p(0)=0$ rules out the latter, so $p\equiv0$.  Hence
$y=0$ and $x=0$.
\end{proof}

For the matrix \(A\) in \eqref{eq:rlt_linear_binary_A}, consider the
equality-generated SDP--RLT relaxation \(\mathcal R(A)\) in \eqref{eq:RA}.
The RLT strengthening $AX=0$ is essential for the high singularity degree
established below.  Indeed, omitting $AX=0$ leaves the basic Shor relaxation
associated with $Ax=0$, whose singularity degree is at most one
\cite{hu2026shorsingularity}.
Here \(r\geq1\), \(A\) has full row rank, \(C\) has full column rank, and
\(\ker A=\range(C)\) by \eqref{eq:rlt_kernel_A_range_C}.  Thus
\(\mathcal R(A)\) and \(\mathcal T(C)\) are precisely the pair of systems
covered by Proposition~\ref{prop:reduced-order-bound}, which gives
\[
    \sd(\mathcal R(A))
    =
    \max\{1,\sd(\mathcal T(C))\}.
\]
Consequently, determining the singularity degree of \(\mathcal R(A)\) reduces
to analyzing the smaller system \(\mathcal T(C)\).

\subsection{Positive semidefinite matrices in the linear span of the reduced
arrow-constraint matrices}

We first characterize the positive semidefinite matrices in the linear span of
the arrow-constraint matrices \(\Phi(c_r(k))\), \(k=1,\ldots,n\), of
\(\mathcal T(C)\).
We state it for every \(s=1,\ldots,r\), because the subsequent
facial-reduction argument successively lowers \(s\).
The specialization of \(\Phi\) to \(c_s(t)\) is the symmetric matrix
polynomial
\begin{equation}\label{eq:rlt_arrow_matrix_polynomial}
\Phi_s(t)
:=\Phi(c_s(t))
=
\begin{pmatrix}
0&-\frac12c_s(t)^T\\[1mm]
-\frac12c_s(t)&c_s(t)c_s(t)^T
\end{pmatrix}
\in\mathbb{S}^{s+1}.
\end{equation}
The entries of $\Phi_s(t)$ involve exactly the consecutive powers
$t,t^2,\ldots,t^{2s}$.

\begin{lem}\label{lem:rlt_arrow_coefficient_span}
For every $s=1,\ldots,r$,
\[
\operatorname{span}\{\Phi_s(k):k=1,\ldots,n\}
\cap\mathbb S^{s+1}_+
=
\{\lambda E_{ss}:\lambda\geq0\}.
\]
Here $E_{ss}=e_se_s^T$ is the diagonal matrix unit in
$\mathbb{S}^{s+1}$, as defined in \eqref{eq:symmetric_basis_matrices}.
\end{lem}

\begin{proof}
Let
\[
W:=\sum_{k=1}^n\mu_k
\Phi_s(k),
\qquad
    a_\ell:=\sum_{k=1}^n\mu_k
    k^\ell,\quad
    \ell=1,\ldots,2s,
\]
and suppose that $W\succeq0$.  Since $W_{00}=0$, positive semidefiniteness
forces row and column $0$ of $W$ to vanish.  The entries of that row are
\[
W_{0j}=-\frac12a_j,
\qquad j=1,\ldots,s,
\]
so
\begin{equation}\label{eq:rlt_hankel_initial_moments}
a_1=\cdots=a_s=0.
\end{equation}
The lower-right block of $W$ is Hankel, with
\begin{equation}\label{eq:rlt_hankel_entries}
W_{ij}=a_{i+j},
\qquad i,j=1,\ldots,s.
\end{equation}
We use the standard zero-propagation argument for positive semidefinite
Hankel matrices; see also \cite[Lemma~4.6]{waki2013facial}.  Suppose for some
$\ell\in\{1,\ldots,s-1\}$ that
$a_1=\cdots=a_{s+\ell-1}=0$.  Then
\[
W_{\ell\ell}=a_{2\ell}=0,
\]
because $2\ell\leq s+\ell-1$.  A zero diagonal entry of a positive
semidefinite matrix forces the corresponding row and column to vanish.  In
particular, $W_{\ell s}=a_{s+\ell}=0$.  If $s=1$,
\eqref{eq:rlt_hankel_initial_moments} already gives the desired conclusion.
If $s\geq2$, starting from \eqref{eq:rlt_hankel_initial_moments} and applying
this argument for $\ell=1,\ldots,s-1$ gives the same conclusion.  Thus, in
either case,
\[
a_1=\cdots=a_{2s-1}=0.
\]
Consequently, $W=a_{2s}E_{ss}$, and $a_{2s}\geq0$ because $W\succeq0$.

Conversely,
the integers \(k=1,\ldots,n\) are distinct and nonzero. Since
$n\geq2r\geq2s$, the matrix
\((k^\ell)_{k=1,\ldots,n,\ \ell=1,\ldots,2s}\)
has full column rank.  Hence there
exist $\nu_1,\ldots,\nu_n$ such that
\[
\sum_{k=1}^n\nu_k
k^\ell
=
\begin{cases}
0,&\ell=1,\ldots,2s-1,\\
1,&\ell=2s.
\end{cases}
\]
It follows from \eqref{eq:rlt_arrow_matrix_polynomial} that
\[
\sum_{k=1}^n\nu_k
\Phi_s(k)
=E_{ss}.
\]
Thus every $\lambda E_{ss}$ with $\lambda\geq0$ belongs to the intersection,
which proves the stated equality.
\end{proof}

\begin{ex}
Let $n=4$ and $s=r=2$.  Then
\[
c_2(t)=\begin{pmatrix}t\\t^2\end{pmatrix},
\qquad
\Phi_2(t)
=
\begin{pmatrix}
0&-\frac12t&-\frac12t^2\\
-\frac12t&t^2&t^3\\
-\frac12t^2&t^3&t^4
\end{pmatrix}.
\]
For
\[
W:=\sum_{k=1}^4\mu_k\Phi_2(k),
\qquad
    a_\ell:=\sum_{k=1}^4\mu_k k^\ell,\quad
    \ell=1,\ldots,4,
\]
we have
\[
W=
\begin{pmatrix}
0&-\frac12a_1&-\frac12a_2\\
-\frac12a_1&a_2&a_3\\
-\frac12a_2&a_3&a_4
\end{pmatrix}.
\]
If $W\succeq0$, then $W_{00}=0$ forces $a_1=a_2=0$.  Since
$W_{11}=a_2=0$, row and column $1$ must also vanish, giving $a_3=0$.
Therefore, $W=a_4E_{22}$ with $a_4\geq0$.  Conversely, direct calculation
gives
\[
E_{22}
=-\frac16\Phi_2(1)+\frac14\Phi_2(2)
-\frac16\Phi_2(3)+\frac1{24}\Phi_2(4).
\]
Thus, this example realizes both inclusions in
\Cref{lem:rlt_arrow_coefficient_span}.
\end{ex}

\subsection{Exact singularity degree of the construction}

For $s=1,\ldots,r$, define
\begin{equation}\label{eq:rlt_reduced_family}
\mathcal T_s
:=
\left\{
Z\in\mathbb{S}^{s+1}_+
\ \middle|\ 
Z_{00}=1,\quad
\langle
\Phi_s(k),Z
\rangle=0,
\ k=1,\ldots,n
\right\}.
\end{equation}
By \eqref{eq:rlt_linear_binary_C}, the rows of \(C\) are
\(c_r(k)^T\).
Comparing \eqref{eq:rlt_reduced_family} with
\eqref{eq:T-general-system} therefore gives
\[
    \mathcal T(C)=\mathcal T_r.
\]
We next compute its singularity degree by analyzing the family
$\{\mathcal T_s\}_{s=1}^r$.

\begin{lem}\label{lem:rlt_reduced_family_degree}
For every $s=1,\ldots,r$,
\begin{equation}\label{eq:rlt_reduced_family_degree}
\sd(\mathcal T_s)=\msd(\mathcal T_s)=s.
\end{equation}
\end{lem}

\begin{proof}
Set $\mathcal T_0:=\{1\}$.  At the first step of any facial-reduction sequence
for $\mathcal T_s$, the exposing matrix has the form
\[
W=\sum_{k=1}^n\mu_k
\Phi_s(k)
\succeq0,
\]
because the multiplier of the normalization must be zero.  By
\Cref{lem:rlt_arrow_coefficient_span}, every such nonzero matrix is a positive
multiple of $E_{ss}$, and $E_{ss}$ itself belongs to this linear span.  This step therefore restricts the problem to the face
$\mathbb{S}^{s+1}_+\cap E_{ss}^{\perp}$, on which row and column $s$ vanish.
Removing this last row and column replaces $c_s(t)$ by $c_{s-1}(t)$ and
reduces $\mathcal T_s$ to $\mathcal T_{s-1}$.  Thus every
facial-reduction sequence makes the same reduction.  Since $\mathcal T_0$ is
strictly feasible in $\mathbb{S}^1_+$, induction proves both equalities in
\eqref{eq:rlt_reduced_family_degree}.
\end{proof}

We now transfer this calculation back to the original SDP--RLT relaxation.

\begin{prop}\label{prop:rlt_linear_binary_high_degree}
For every $n\geq 2$, let \(A\) be the matrix defined in
\eqref{eq:rlt_linear_binary_A}.  Then
\[
\sd(\mathcal R(A))=r=\left\lfloor\frac n2\right\rfloor.
\]
\end{prop}

\begin{proof}
The reduced system satisfies \(\mathcal T(C)=\mathcal T_r\).  Therefore,
\Cref{prop:reduced-order-bound,lem:rlt_reduced_family_degree} give
\[
    \sd(\mathcal R(A))
    =\max\{1,\sd(\mathcal T_r)\}
    =\max\{1,r\}
    =r.
\]
\end{proof}

For every $n\geq2$, the construction uses $n$ binary variables
and an $(n+1)\times(n+1)$ matrix, while its singularity degree is
\[
r=\left\lfloor\frac n2\right\rfloor.
\]
The point is nontrivial at the formulation level: although the binary set is
defined solely by linear equalities, the standard equality-generated SDP--RLT
strengthening produces a
facial-reduction sequence whose length grows linearly
with the problem dimension.  Thus high singularity degree, a structural
source of weak error bounds and numerical sensitivity, can arise in SDP
relaxations of linearly constrained binary sets.

The construction above supplies the lower bound \(\lfloor n/2\rfloor\).  The
next section proves that this value is globally optimal within the
equality-generated SDP--RLT class.

\section{The sharp upper bound for the equality-generated SDP--RLT class}
\label{sec:structured_upper_bound}

This section proves that, for every full-column-rank
\(C\in\R^{n\times r}\), the maximum singularity degree of
\(\mathcal T(C)\) is at most \(\min\{n-r,r\}\), and hence so is its
singularity degree.  Corollary~\ref{cor:global-rlt-bound} then resolves the
extremal problem stated in \eqref{eq:intro-extremal-problem}.

For a full-column-rank matrix, we call the number of rows minus the number of
columns its \emph{row redundancy}.  It is the number of rows beyond the
minimum required for full column rank.  In particular, the row redundancy of
\(C\) is \(m:=n-r\).

The case \(r=0\) is immediate because \(\mathcal T(C)=\{1\}\).  For
\(r\geq1\), Proposition~\ref{prop:reduced-order-bound} gives the upper bound
\(\msd(\mathcal T(C))\leq r\), so the main task is to prove the complementary
bound by the row redundancy \(m\).  The proof has two ingredients.
Subsection~\ref{subsec:first-step-support-row-redundancy} analyzes the first
facial-reduction step and derives the row-redundancy bound needed for the
induction.
Subsection~\ref{subsec:auxiliary-restricted-system} introduces an auxiliary
restriction and bounds the number of steps that disappear under this
restriction.
Subsection~\ref{subsec:sharp-upper-bound} combines these ingredients by
induction on row redundancy.

\subsection{The first facial-reduction step and its row-redundancy bound}
\label{subsec:first-step-support-row-redundancy}

Let \(C\in\R^{n\times r}\) have full column rank, and set \(m:=n-r\).
Since \(E_{00}\in\mathcal T(C)\), orthogonality to the affine
constraint set forces the multiplier of the normalization equation
\(Z_{00}=1\) in any exposing matrix to be zero.
Consider the exposing matrix at the first step of an arbitrary
facial-reduction sequence,
\[
    W:=\sum_{i=1}^n\mu_i\Phi(c_i)
    =
    \begin{pmatrix}
        0&-\frac12\mu^TC\\
        -\frac12C^T\mu&C^T\Diag(\mu)C
    \end{pmatrix}
    \succeq0,
\]
where \(W\neq0\).  Since \(W_{00}=0\), positive semidefiniteness forces the
zeroth row and column of \(W\) to vanish.  Hence
\[
    C^T\mu=0,
\]
and therefore
\[
    W=
    \begin{pmatrix}0&0\\0&B\end{pmatrix},
    \qquad
    B:=C^T\Diag(\mu)C\succeq0.
\]
Thus the special structure of \(\Phi(c_i)\) makes the first exposing matrix
simultaneously produce a linear dependence among the rows of \(C\) and a
positive semidefinite matrix \(B\) determining the exposed face.  Set
\[
    \rho:=\rank(W)=\rank(B),\qquad k:=r-\rho.
\]
Let the columns of \(\Gamma\in\R^{r\times k}\) span \(\ker B\).  Set
\[
    D:=C\Gamma
    =
    \begin{pmatrix}d_1^T\\ \vdots\\ d_n^T\end{pmatrix},
    \qquad
    V:=\begin{pmatrix}1&0\\0&\Gamma\end{pmatrix}.
\]
By the characterization of faces of the positive semidefinite cone in
Section~\ref{subsec:facial_reduction_algorithm}, \(W\) exposes the face
associated with \(\spanop\{e_0\}\mathbin{\oplus}\ker B\).  Moreover,
\[
    V^T\Phi(c_i)V=\Phi(d_i),
    \qquad i=1,\ldots,n.
\]
Consequently,
\begin{equation}\label{eq:first-step-face-parametrization}
    \mathcal T(C)
    =
    \left\{
        VZV^T
        \ \middle|\
        Z\in\mathcal T(D)
    \right\}.
\end{equation}
Thus \(\mathcal T(D)\) is precisely the formulation obtained by
parametrizing the face exposed at this first facial-reduction step.

To state the row-redundancy bound, let
\(S:=\{i\in\{1,\ldots,n\}:\mu_i\neq0\}\) be the index set of the nonzero
entries of the multiplier vector \(\mu\).  Set
\(S^c:=\{1,\ldots,n\}\setminus S\).  Let \(D_S\) denote the submatrix of
\(D\) formed by the rows indexed by \(S\).  Define \(p:=|S|\), let
\(\mathbf1\in\R^p\) denote the all-ones
vector, and let \(\mu_S\in\R^p\) denote the subvector of \(\mu\) indexed by
\(S\).  Set
\[
    s:=\rank(D_S),\quad
    h:=\rank\begin{pmatrix}\mathbf1&D_S\end{pmatrix}.
\]
Since \(D_S\) has rank \(s\), exactly \(k-s\) additional independent rows
are needed to span the \(k\)-dimensional row space of \(D\).  Among the
\(|S^c|\) available rows, this leaves
\[
    \eta:=|S^c|-(k-s)
\]
remaining rows after a rank-extending subset has been selected.  We use this
notation in the lemma below and in the proof of
Theorem~\ref{thm:structured-upper}.

\begin{lem}\label{lem:first-stress-support}
Under the first-step notation above,
\begin{equation}\label{eq:first-stress-support-bound}
    p=|S|\geq h+s+\rho.
\end{equation}
Equivalently,
\begin{equation}\label{eq:first-stress-row-excess}
    0\leq \eta\leq m-h.
\end{equation}
In particular, \(h+\eta\leq m\).
\end{lem}

\begin{proof}
Choose \(R\in\R^{r\times\rho}\) whose columns span \(\range(B)\).  Since
\(B\) is symmetric, \(\range(B)=(\ker B)^\perp\), so
\(Q=(\Gamma\ R)\) is
invertible.  Set \(Z:=CR\), and let \(Z_S\) denote the submatrix of \(Z\)
consisting of the rows indexed by \(S\).  Then
\[
    CQ=(D\ Z).
\]
Since \(B\succeq0\) and \(\range(R)=\range(B)\),
\(P:=R^TBR\in\mathbb S^\rho\) is positive definite.  Moreover,
\(B\Gamma=0\), so the change of basis \(Q=(\Gamma\ R)\) gives
\[
    Q^TBQ
    =
    \begin{pmatrix}
        \Gamma^TB\Gamma&\Gamma^TBR\\
        R^TB\Gamma&R^TBR
    \end{pmatrix}
    =
    \begin{pmatrix}
        0&0\\
        0&P
    \end{pmatrix}.
\]
Let
\[
    \widetilde D_S:=\begin{pmatrix}\mathbf1&D_S\end{pmatrix}.
\]
Put \(J=\Diag(\mu_S)\), which is nonsingular, and define
\[
    \mathcal E:=\range(\widetilde D_S),
    \qquad
    \mathcal G:=\range\begin{pmatrix}D_S&Z_S\end{pmatrix}.
\]
By definition, \(\dim\mathcal E=h\).

Because \(\mu\) is supported on \(S\) and \(J\mathbf1=\mu_S\), the equation
\(C^T\mu=0\) gives
\begin{equation}\label{eq:first-stress-linear-identities}
    \mathbf1^TJD_S=0,
    \qquad
    \mathbf1^TJZ_S=0.
\end{equation}
Moreover,
\begin{equation}\label{eq:first-stress-block-identities}
    \begin{pmatrix}
        D_S^TJD_S&D_S^TJZ_S\\
        Z_S^TJD_S&Z_S^TJZ_S
    \end{pmatrix}
    =
    (CQ)^T\Diag(\mu)(CQ)
    =
    Q^TBQ
    =
    \begin{pmatrix}
        0&0\\
        0&P
    \end{pmatrix}.
\end{equation}
The \((2,2)\) block of \eqref{eq:first-stress-block-identities} gives
\(Z_S^TJZ_S=P\succ0\), so \(Z_S\) has full column rank \(\rho\).  Moreover,
suppose \(w=D_Sa=Z_Sb\).  The \((1,1)\) and \((2,2)\) blocks of
\eqref{eq:first-stress-block-identities} give, respectively,
\[
    w^TJw=a^TD_S^TJD_Sa=0,
\]
and
\[
    w^TJw=b^TZ_S^TJZ_Sb=b^TPb.
\]
Since \(P\succ0\), we have \(b=0\), and hence \(w=0\).  Consequently,
\[
    \range(D_S)\cap\range(Z_S)=\{0\}.
\]
Since \(\rank(D_S)=s\), it follows that
\[
    \dim\mathcal G=s+\rho.
\]

The identities in \eqref{eq:first-stress-linear-identities}, together with
the \((1,1)\) and \((1,2)\) blocks of
\eqref{eq:first-stress-block-identities}, combine into
\[
    \widetilde D_S^TJ
    \begin{pmatrix}D_S&Z_S\end{pmatrix}
    =0.
\]
Thus, with respect to the ordinary Euclidean inner product on \(\R^p\),
\[
    \mathcal E\perp J\mathcal G.
\]
Since \(J\) is nonsingular, \(\dim(J\mathcal G)=s+\rho\).  The two
orthogonal subspaces must fit in \(\R^p\), and hence
\[
    h+s+\rho
    =\dim\mathcal E+\dim(J\mathcal G)
    \leq p.
\]
This proves \eqref{eq:first-stress-support-bound}.  Equivalently,
\(s+\rho\leq p-h\).

By its definition, \(\eta\geq0\).  Since \(n=m+r=k+\rho+m\),
\[
    \eta=(n-p)-(k-s)=s+\rho+m-p\leq m-h.
\]
\end{proof}

\subsection{The auxiliary restricted system}
\label{subsec:auxiliary-restricted-system}

The proof of the upper bound below uses two successive operations, which have
different roles.  After this first facial-reduction step for \(\mathcal T(C)\),
parametrizing the exposed face with \(V=\Diag(1,\Gamma)\) gives
\[
    C\ \xrightarrow{\;\Gamma\;}\ D=C\Gamma;
\]
this parametrizes the face reached by the facial-reduction step, and the
remaining steps form a facial-reduction sequence for \(\mathcal T(D)\).  To
bound its length, we restrict \(\mathcal T(D)\) to the face associated with
\(\spanop\{e_0\}\mathbin{\oplus}G\), for a chosen subspace
\(G\subseteq\R^k\).  If the columns of \(U\) form an orthonormal basis of
\(G\), then, in this basis, the restricted system is \(\mathcal T(E)\), where
\[
    D\ \xrightarrow{\;U\;}\ E=DU.
\]
This is the restriction construction described in
Subsection~\ref{subsec:restriction-psd-face}, specialized to the structured
system \(\mathcal T(D)\); it is not another facial-reduction step.  This
restriction makes specified rows of \(D\) identical.  The lemma below bounds
the length of a facial-reduction sequence
for \(\mathcal T(D)\) in terms of that of the restricted system
\(\mathcal T(E)\).

\begin{lem}[Auxiliary restriction lemma]\label{lem:auxiliary-restriction}
Let \(D\in\R^{n\times k}\) have full column rank.
Let \(G\subseteq\R^k\) be a subspace, set \(q:=\dim G\), and choose
\(U\in\R^{k\times q}\) with orthonormal columns spanning \(G\).  Define
\[
    E:=DU\in\R^{n\times q}.
\]
Then \(E\) has full column rank, and
\[
    \msd(\mathcal T(D))
    \leq
    \dim(G^\perp)+\msd(\mathcal T(E)).
\]
\end{lem}

\begin{proof}
Since \(D\) has full column rank and \(U\) has full column
rank, \(E=DU\) has full column rank. Set \(\widehat U:=\Diag(1,U)\), which has orthonormal columns, and let \(F_G\)
be the face associated with
\(\spanop\{e_0\}\mathbin{\oplus}G=\range(\widehat U)\).
Since \(E_{00}\in\mathcal T(D)\cap F_G\), this intersection is nonempty.
Applying the same
block-diagonal congruence calculation used in deriving
\eqref{eq:first-step-face-parametrization}, with \(U\) in place of \(\Gamma\),
gives
\[
    \mathcal T(D)\cap F_G
    =
    \{\widehat UZ\widehat U^T:Z\in\mathcal T(E)\}.
\]
Since
\((\spanop\{e_0\}\mathbin{\oplus}G)^\perp=\{0\}\mathbin{\oplus}G^\perp\),
the subspace associated with \(F_G\) has orthogonal complement of dimension
\(\dim(G^\perp)\).  Applying
Lemma~\ref{lem:msd-restriction-bound} to the affine system defining
\(\mathcal T(D)\) now gives the claimed inequality.
\end{proof}

\subsection{The sharp upper bound}
\label{subsec:sharp-upper-bound}

Lemmas~\ref{lem:first-stress-support} and
\ref{lem:auxiliary-restriction} provide
the two ingredients for the induction.  The former gives the row-redundancy
bound for the first facial-reduction step, while the latter bounds the number
of steps that disappear under the auxiliary restriction.  We now
combine them to prove the sharp upper bound.

The use of maximum singularity degree in the theorem below deserves a brief
explanation.  Although our final target is \(\sd(\mathcal T(C))\), the
preceding restriction argument must control arbitrary facial-reduction
sequences, not only a shortest one.  By
Lemma~\ref{lem:restrict-fr-sequence}, restricting the feasible set to a face
of the positive semidefinite cone may turn some steps into repetitions, but
the remaining strict steps form a partial facial-reduction sequence for the
restricted problem.  This partial sequence can be extended to a
facial-reduction sequence, so its length
is bounded by maximum singularity degree, whereas it need not be bounded by
singularity degree.  This is why the theorem establishes the stronger bound
on maximum singularity degree.

\begin{thm}\label{thm:structured-upper}
For every full-column-rank \(C\in\R^{n\times r}\),
\[
    \msd(\mathcal T(C))
    \leq
    \min\{n-r,r\}.
\]
\end{thm}

\begin{proof}
If \(r=0\), then \(C\) has no columns and \(\mathcal T(C)=\{1\}\), so its
maximum singularity degree is zero.  Assume henceforth that \(r\geq1\).

The row redundancy of \(C\) is \(m:=n-r\).
Proposition~\ref{prop:reduced-order-bound} already gives
\(\msd(\mathcal T(C))\leq r\), so it remains to prove
\begin{equation}\label{eq:all-chain-bound}
    \msd(\mathcal T(C))\leq m
\end{equation}
by strong induction on the row redundancy \(m\).

If \(m=0\), then \(C\) is square and nonsingular.  A first exposing
multiplier must satisfy \(C^T\mu=0\), because the exposing matrix has zero
zeroth row.  Thus \(\mu=0\), and there is no facial-reduction step.

Let \(m\geq1\), and assume that the maximum singularity degree is at most the
row redundancy for every full-column-rank matrix whose row redundancy is
smaller than \(m\).  Consider an arbitrary facial-reduction sequence of
length \(d\) for \(\mathcal T(C)\).  If \(d=0\), there is nothing to prove.
Assume \(d\geq1\), and apply the first-step notation introduced before
Lemma~\ref{lem:first-stress-support} to the first step of this sequence.
By \eqref{eq:first-step-face-parametrization}, the remaining \(d-1\) steps
form a facial-reduction sequence for \(\mathcal T(D)\).

Recall that \(D\in\R^{n\times k}\) and
\[
    s=\rank(D_S),
    \qquad
    h=\rank\begin{pmatrix}\mathbf1&D_S\end{pmatrix}.
\]
Thus \(s\) is the dimension of the linear span of \(\{d_i:i\in S\}\), whereas
\(h-1\) is the dimension of their affine hull.  Let \(G\) be the orthogonal
complement of the direction space of the affine
hull of \(\{d_i:i\in S\}\); that is,
\[
    G:=\bigl(\spanop\{d_i-d_j:i,j\in S\}\bigr)^\perp.
\]
Consequently,
\[
    \dim(G^\perp)=h-1,\qquad \dim G=k-h+1.
\]
Apply the auxiliary restriction of
Subsection~\ref{subsec:auxiliary-restricted-system} to the present \(D\) and
\(G\).  Set
\(q:=\dim G=k-h+1\), choose
\(U\in\R^{k\times q}\) with orthonormal columns spanning \(G\), and define
\[
    E:=DU
    \in\R^{n\times q},
    \qquad
    \bar d_i:=U^Td_i,
    \quad i=1,\ldots,n,
\]
so the \(i\)th row of \(E\) is \(\bar d_i^T\).  Since every
\(d_i-d_j\), \(i,j\in S\), belongs to \(G^\perp\),
\[
    \bar d_i-\bar d_j=U^T(d_i-d_j)=0.
\]
Thus the rows of \(E\) indexed by \(S\) have a common value \(a^T\).
Geometrically, the projection onto \(G\) either sends all points \(d_i\),
\(i\in S\), to zero, when their affine hull contains the origin, or sends them
all to the same nonzero point, when their affine hull does not contain the
origin.  This is precisely the dichotomy
\[
    s=h-1\quad\text{and}\quad a=0,
    \qquad\text{or}\qquad
    s=h\quad\text{and}\quad a\neq0.
\]

If \(s=h-1\), then \(a=0\), so the rows of \(E\) indexed by \(S\) are zero,
and we delete them.  If \(s=h\), then \(a\neq0\), so these rows are identical
and nonzero, and we retain one copy.
Denote the resulting matrix by \(\widehat E\).
Because \(\Phi(0)=0\) and repeated rows give identical constraint matrices,
deleting these equations does not change the affine constraint system.  Hence
\(\mathcal T(E)\) and \(\mathcal T(\widehat E)\) have the same exposing
matrices and facial-reduction sequences; in particular,
\[
    \msd(\mathcal T(E))=\msd(\mathcal T(\widehat E)).
\]
The matrix \(E=DU\) has full column rank, and deleting zero rows and repeated
copies does not change its row span.  Hence \(\widehat E\) has full column
rank \(q=\dim G\).  Recall from the definition of \(\eta\) that
\(\lvert S^c\rvert=k-s+\eta\).  Let \(n_E\) denote the number of rows of
\(\widehat E\).  If \(s=h-1\),
\[
    n_E=|S^c|
    =k-s+\eta
    =\dim G+\eta.
\]
If \(s=h\),
\[
    n_E=1+|S^c|
    =1+k-s+\eta
    =\dim G+\eta.
\]
Thus, in either case, the row redundancy of \(\widehat E\) is \(\eta\).
Since \(h\geq1\), Lemma~\ref{lem:first-stress-support} gives
\[
    0\leq\eta\leq m-h<m.
\]
By Lemma~\ref{lem:auxiliary-restriction} and the induction hypothesis applied to
\(\widehat E\), which has row redundancy \(\eta<m\), we obtain
\[
    \msd(\mathcal T(D))
    \leq
    \dim(G^\perp)+\msd(\mathcal T(E))
    =
    (h-1)+\msd(\mathcal T(\widehat E))
    \leq
    h-1+\eta.
\]
Including the first step and using
\eqref{eq:first-stress-row-excess} gives
\[
    \msd(\mathcal T(C))
    \leq
    1+(h-1)+\eta
    =h+\eta
    \leq m.
\]
This proves \eqref{eq:all-chain-bound} and hence the theorem.
\end{proof}

Since the singularity degree is no larger than the maximum singularity degree,
Theorem~\ref{thm:structured-upper} also gives
\[
    \sd(\mathcal T(C))\leq\min\{n-r,r\}.
\]

\begin{cor}\label{cor:global-rlt-bound}
The worst-case singularity degree over all equality-generated SDP--RLT
relaxations \(\mathcal R(A,b)\) in \(n\) binary variables for which
\(P\neq\emptyset\) is
\[
    \max_{A,b:\,P\neq\emptyset}\sd(\mathcal R(A,b))
    =
    \begin{cases}
        1,&n=1,\\[1mm]
        \lfloor n/2\rfloor,&n\geq2.
    \end{cases}
\]
\end{cor}

\begin{proof}
By \Cref{lem:affine-to-homogeneous}, it suffices to consider \(b=0\).
After removing redundant rows as in
Subsection~\ref{sec:equality-generated-rlt-relaxation}, assume that \(A\) has
full row rank, and put \(m:=\rank(A)\) and \(r:=n-m\).  If \(m=0\), the
basic arrow system is strictly feasible and has singularity degree zero.  If
\(r=0\), the feasible set is \(\{E_{00}\}\), whose minimal face is the ray
generated by \(E_{00}\), exposed in one step.  For \(n=1\), these
are the only two cases, and \(A=[1]\), \(b=0\) gives a relaxation with
singularity degree one.  Assume henceforth that \(n\geq2\), \(m>0\), and
\(r>0\).  Then
Proposition~\ref{prop:reduced-order-bound},
Theorem~\ref{thm:structured-upper}, and \eqref{eq:sd-le-msd} give
\[
    \sd(\mathcal R(A))
    =
    \max\{1,\sd(\mathcal T(C))\}
    \leq
    \min\{n-r,r\}.
\]
The two boundary cases have singularity degree at most one, while
\[
    \max_{1\leq r\leq n-1}\min\{n-r,r\}
    =\left\lfloor\frac n2\right\rfloor.
\]
The construction in
Section~\ref{sec:binary_high_singularity_degree} attains this value.
\end{proof}

\subsection{Attainment for every parameter pair}

The preceding corollary maximizes over \(r\).  The upper bound in
Theorem~\ref{thm:structured-upper} is, in fact, attained for each fixed pair
\((n,r)\).

\begin{cor}\label{cor:structured-bound-attainment}
For every integer \(n\geq1\) and every \(r\) with \(0\leq r\leq n\), there is
a full-column-rank matrix \(C\in\R^{n\times r}\) such that
\[
    \sd(\mathcal T(C))
    =
    \msd(\mathcal T(C))
    =
    \min\{n-r,r\}.
\]
\end{cor}

\begin{proof}
If \(r=0\), then \(\mathcal T(C)=\{1\}\).  If \(r=n\), take \(C=I_n\); the
resulting basic arrow system is strictly feasible.  Thus both singularity
degrees are zero in either boundary case.  Suppose that \(0<r<n\), and
set \(m:=n-r\).  If \(n\geq2r\), the consecutive-power construction in
Section~\ref{sec:binary_high_singularity_degree} has singularity degree
\(r\).  More explicitly, \(n\geq2r\) implies
\(r\leq\lfloor n/2\rfloor\), and taking \(s=r\) in
Lemma~\ref{lem:rlt_reduced_family_degree} gives the matrix
\(C_{ij}=i^j\), \(i=1,\ldots,n\), \(j=1,\ldots,r\), with singularity
degree \(r\).  Since \(n-r\geq r\), this attains the bound \(\min\{n-r,r\}\).

It remains to consider \(m<r\).  For these fixed \(m\) and \(r\), let
\(C_{\rm pow}\in\R^{2m\times m}\) be the consecutive-power matrix defined in
\eqref{eq:rlt_linear_binary_C}; explicitly,
\((C_{\rm pow})_{ij}=i^j\) for \(i=1,\ldots,2m\) and
\(j=1,\ldots,m\).  Define
\[
    C
    :=
    \begin{pmatrix}
        C_{\rm pow}&0\\
        0&I_{r-m}
    \end{pmatrix}
    \in\R^{n\times r}.
\]
Indeed, \(C\) has
\[
    2m+(r-m)=m+r=n
\]
rows and \(m+(r-m)=r\) columns, and it has full column rank.

The last \(r-m\) rows are the standard unit vectors
\(e_{m+1}^T,\ldots,e_r^T\) in \(\R^r\).  If \(\nu_j\) is the multiplier
associated with the constraint matrix \(\Phi(e_j)\), then the \((0,j)\) entry
of the exposing matrix is
\(-\nu_j/2\), because no other row of \(C\) has a nonzero \(j\)th component.
At the first step, the exposing matrix is
positive semidefinite, and its zero \((0,0)\) entry forces its zeroth row to
vanish; hence \(\nu_j=0\).  The same argument applies at each subsequent
step: the subspace associated with the current face still contains
\(\spanop\{e_0,e_{m+1},\ldots,e_r\}\).  Therefore, the principal submatrix of
the next exposing matrix indexed by \(0,m+1,\ldots,r\) is positive
semidefinite.  Its \((0,0)\) entry is zero, so its zeroth row vanishes and
\(\nu_j=0\) for every \(j=m+1,\ldots,r\).  Thus the appended rows do not
participate in facial reduction, and the first \(m\) coordinates reproduce
the forced \(m\)-step sequence of the consecutive-power construction.  After
these steps, the basic arrow system in the remaining \(r-m\) coordinates is
strictly feasible.
Thus the resulting system has singularity degree \(m=n-r\).  Since
\(m\leq r\), this again attains the bound \(\min\{n-r,r\}\).
In both cases, Theorem~\ref{thm:structured-upper} and the fact that singularity
degree is no larger than maximum singularity degree force maximum singularity
degree to attain the same value.
\end{proof}

For \(0<r<n\), choose a full-row-rank matrix \(A\) satisfying
\(\ker(A)=\range(C)\).  Then \(P(A,0)\neq\emptyset\), and
Proposition~\ref{prop:reduced-order-bound} gives
\(\sd(\mathcal R(A,0))=\min\{n-r,r\}\).  Hence the rank--nullity bound is
attained for every pair with \(0<r<n\).
At the boundary, \(r=n\) gives a strictly feasible basic
arrow system with singularity degree zero, whereas \(r=0\) gives the ray
generated by \(E_{00}\), whose singularity degree is one.

The construction for \(m<r\) is used only to establish sharpness within
the class \(\mathcal T(C)\); unlike the construction in
Section~\ref{sec:binary_high_singularity_degree}, it need not define a
singleton binary set.

\section{Conclusion}

This paper determines the exact worst-case singularity degree of the
equality-generated SDP--RLT relaxations of nonempty binary sets defined by
\(Ax=b\).  If \(\rank(A)=m\) and \(0<m<n\), then the associated relaxation has
singularity degree at most \(\min\{m,n-m\}\), and this bound is attained for
every such rank--nullity pair.  Consequently, the worst-case singularity
degree over all these relaxations is \(1\) for \(n=1\) and
\(\lfloor n/2\rfloor\) for \(n\geq2\).  For \(n\geq2\), the latter value is
attained even when the binary feasible set is a singleton.

Unlike singularity degree itself, the rank and nullity of the equality system
are available directly from the formulation.  The bound therefore certifies in
advance that systems with either few independent equalities or small nullity
have small singularity degree; the largest worst-case degree occurs only when
rank and nullity are nearly balanced.  Through general SDP error-bound theory,
the rank--nullity bound yields a more favorable H\"older exponent in estimates
of the distance to feasibility from constraint residuals.

\appendix

\section{Linear span of the lifted equality-constraint matrices}
\label{app:rlt_equality_constraint_span}

This appendix proves the identity for the linear span of the lifted
equality-constraint matrices used in
Proposition~\ref{prop:reduced-order-bound}.  We first state the underlying
linear-algebra identity in dimensions matching the application and then apply
it to the lifted equations \(Ax=0\) and \(AX=0\).

\begin{prop}\label{prop:constraint-map-range}
Let \(V\in\R^{(n+1)\times(r+1)}\) have full column rank, and let
\(\widetilde A\in\R^{m\times(n+1)}\) have full row rank with
\(\ker(\widetilde A)=\range(V)\).  Define
\[
    \mathcal B:\mathbb S^{n+1}\to\mathbb R^{(n+1)\times m},
    \qquad
    \mathcal B(Y):=Y\widetilde A^T.
\]
Then
\[
    \range(\mathcal B^*)
    =\{N\in\mathbb S^{n+1}\mid V^TNV=0\}.
\]
\end{prop}

\begin{proof}
Since \(\range(\widetilde A^T)=\ker(\widetilde A)^\perp
=\range(V)^\perp\) and \(Y\) is symmetric,
\[
    \ker(\mathcal B)
    =\{VZV^T\mid Z\in\mathbb S^{r+1}\}.
\]
Indeed, \(Y\widetilde A^T=0\) is equivalent to
\(\range(V)^\perp\subseteq\ker(Y)\), and hence to
\(\range(Y)\subseteq\range(V)\).  Therefore
\[
    \range(\mathcal B^*)
    =\ker(\mathcal B)^\perp
    =\{N\in\mathbb S^{n+1}\mid V^TNV=0\},
\]
as claimed.
\end{proof}

\begin{cor}\label{cor:lifted-equality-constraint-span}
Let \(A\in\R^{m\times n}\) have full row rank, let
\(C\in\R^{n\times r}\) have full column rank with
\(\ker(A)=\range(C)\), and set
\[
    V:=\begin{pmatrix}1&0\\0&C\end{pmatrix}.
\]
If \(L_A\) is the linear span of the constraint matrices for the lifted
equations \(Ax=0\) and \(AX=0\), then
\[
    L_A
    =\{N\in\mathbb S^{n+1}\mid V^TNV=0\}.
\]
\end{cor}

\begin{proof}
Set \(\widetilde A:=\begin{pmatrix}0&A\end{pmatrix}\).  The lifted equations
are precisely \(Y\widetilde A^T=0\), so their coefficient matrices span the
range of the adjoint of \(Y\mapsto Y\widetilde A^T\).  Moreover,
\[
    \ker(\widetilde A)
    =\mathbb R\oplus\ker(A)
    =\mathbb R\oplus\range(C)
    =\range(V).
\]
The result now follows from \Cref{prop:constraint-map-range}.
\end{proof}

\bibliographystyle{siam}
\bibliography{mybib}

\begin{thebibliography}{10}

\bibitem{borwein1981regularizing}
{\sc J.~Borwein and H.~Wolkowicz}, {\em Regularizing the abstract convex
  program}, Journal of Mathematical Analysis and Applications, 83 (1981),
  pp.~495--530.

\bibitem{borwein1981facial}
{\sc J.~M. Borwein and H.~Wolkowicz}, {\em Facial reduction for a cone-convex
  programming problem}, Journal of the Australian Mathematical Society, 30
  (1981), pp.~369--380.

\bibitem{drusvyatskiy2017many}
{\sc D.~Drusvyatskiy and H.~Wolkowicz}, {\em The many faces of degeneracy in
  conic optimization}, Foundations and Trends{\textregistered} in Optimization,
  3 (2017), pp.~77--170.

\bibitem{hou2026lowrankalm}
{\sc D.~{Hou}, T.~{Tang}, and K.~{Toh}}, {\em {A Low-Rank Augmented Lagrangian
  Method for Polyhedral-SDP and Moment-SOS Relaxations of Polynomial
  Optimization}}, Mathematical Programming,  (2026).

\bibitem{hu2024maximum}
{\sc H.~Hu}, {\em The maximum singularity degree for linear and semidefinite
  programming}, Preprint arXiv:2402.11795,  (2024).

\bibitem{hu2026shorsingularity}
\leavevmode\vrule height 2pt depth -1.6pt width 23pt, {\em {On the Singularity
  Degree of Shor Relaxations for $0$--$1$ Programs}}.
\newblock {arXiv:2607.12476 [math.OC]}, {2026}.

\bibitem{lourencco2018facial}
{\sc B.~F. Louren{\c{c}}o, M.~Muramatsu, and T.~Tsuchiya}, {\em Facial
  reduction and partial polyhedrality}, SIAM Journal on Optimization, 28
  (2018), pp.~2304--2326.

\bibitem{pataki2013strong}
{\sc G.~Pataki}, {\em Strong duality in conic linear programming: facial
  reduction and extended duals}, Proceedings of Jonfest: A conference in honour
  of the 60th birthday of Jon Borwein,  (2013), pp.~613--634.

\bibitem{sherali1990hierarchy}
{\sc H.~D. Sherali and W.~P. Adams}, {\em A hierarchy of relaxations between
  the continuous and convex hull relaxations for 0--1 programming}, SIAM
  Journal on Discrete Mathematics, 3 (1990), pp.~411--430.

\bibitem{sherali2013reformulation}
\leavevmode\vrule height 2pt depth -1.6pt width 23pt, {\em A
  Reformulation-Linearization Technique for Solving Discrete and Continuous
  Nonconvex Problems}, Springer Science \& Business Media, 2013.

\bibitem{sremac2017complete}
{\sc S.~Sremac, H.~Woerdeman, and H.~Wolkowicz}, {\em Complete facial reduction
  in one step for spectrahedra}, arXiv preprint arXiv:1710.07410,  (2017).

\bibitem{sremac2021error}
{\sc S.~Sremac, H.~J. Woerdeman, and H.~Wolkowicz}, {\em Error bounds and
  singularity degree in semidefinite programming}, SIAM Journal on
  Optimization, 31 (2021), pp.~812--836.

\bibitem{sturm2000error}
{\sc J.~F. Sturm}, {\em Error bounds for linear matrix inequalities}, SIAM
  Journal on Optimization, 10 (2000), pp.~1228--1248.

\bibitem{tanigawa2017singularity}
{\sc S.-i. Tanigawa}, {\em Singularity degree of the positive semidefinite
  matrix completion problem}, SIAM Journal on Optimization, 27 (2017),
  pp.~986--1009.

\bibitem{tuncel2001slater}
{\sc L.~Tun{\c{c}}el}, {\em On the {S}later condition for the {SDP} relaxations
  of nonconvex sets}, Operations Research Letters, 29 (2001), pp.~181--186.

\bibitem{waki2013facial}
{\sc H.~Waki and M.~Muramatsu}, {\em Facial reduction algorithms for conic
  optimization problems}, Journal of Optimization Theory and Applications, 158
  (2013), pp.~188--215.

\bibitem{yildirim2026relaxations}
{\sc E.~{Y{\i}ld{\i}r{\i}m}}, {\em {Relaxations of KKT Conditions Do Not
  Strengthen Finite RLT and SDP-RLT Bounds for Nonconvex Quadratic Programs}},
  Journal of Global Optimization, 94 (2026), pp.~891--918.

\bibitem{ziegler2000zeroone}
{\sc G.~M. Ziegler}, {\em Lectures on 0/1-polytopes}, in
  Polytopes---Combinatorics and Computation, vol.~29 of DMV Seminar,
  Birkh{\"a}user, Basel, 2000, pp.~1--41.

\end{thebibliography}
\addcontentsline{toc}{section}{Bibliography}

\end{document}